\documentclass[12pt,a4paper]{article}
\usepackage{url}
\usepackage{amsmath}
\usepackage{amssymb,latexsym}
\usepackage[mathscr]{eucal}
\usepackage[makeroom]{cancel}
\usepackage[dvips]{graphics}
\usepackage[pdftex]{graphicx}
\usepackage[T1]{fontenc}
\usepackage[utf8]{inputenc}
\usepackage[right]{lineno}
\usepackage{amscd,amsthm,amsfonts,amsbsy,color}
\usepackage{enumerate,epstopdf,times,cite}

\theoremstyle{plain}
\newtheorem{theorem}{Theorem}

\newtheorem{definition}{Definition}

\newtheorem{lemma}{Lemma}
\newtheorem{proposition}{Proposition}
\newtheorem{remark}{Remark}

\begin{document}
\title{A modified proximal contraction principle with applications to variational
inequality problems}
\author{Aftab Alam}

\maketitle
\begin{center}
{\footnotesize \footnotesize Formerly at: Department of Mathematics, Aligarh Muslim University, Aligarh-202002, India.\\
Email Id:  aafu.amu@gmail.com\\
Orcid Id: 0000-0002-0479-9118\\
*Corresponding author: aafu.amu@gmail.com}
\end{center}

{\footnotesize{\noindent {\bf Abstract.}}  In this paper, we introduce the notions of proximally  completeness, proximally  closedness and proximally continuity and utilize the same  to prove a result on existence and uniqueness of best proximity points in the setting of metric space (not necessarily complete). Our newly proved result enriches, sharpens, improves and modifies the proximal contraction principle of  Basha [J. Optim. Theory Appl. 2011:151 (2011), 210-216]. In order to illustrate the effectiveness of our finding, we discuss the sufficient conditions ensuring the existence of a unique solution
of certain variational inequality problem.\\


}

\section{Introduction}
Consider the non-self-mapping $T:A\rightarrow B$, whereas $A$ and $B$ stand for two nonempty subset a metric space $(X,d)$.  We say that an element $x\in A$ is a fixed point of the mapping if $T(x)=x$ provided $A\cap T(A)\neq\emptyset$. On the other hand, in case of $A\cap T(A)=\emptyset$, $T$ has no fixed point so that for each $x\in A$, $d(x,Tx)>0$. In such a case, we are interesting to compute an optimal approximate solution of functional equation $T(x)=x$ such that $d(x,Tx)$ is closet to zero. In this perspective, best approximation theory and best proximity point analysis have been appeared. \\

In $1969$,  Fan \cite{BAP-1} proved the first best approximation theorem in the setting of Hausdorff locally convex topological vector space.  In the subsequent years, the classical best approximation theorem of  Fan has been generalized and extended by many researchers, but we merely refer \cite{BAP-2,BAP-3,BAP-4,BAP-5} and references therein. A comprehensive and unified approach to such best approximation theorems has been furnished by Vetrivel et al. \cite{BAP-6}. On the other hand, Eldred and Veeramani \cite{BPP-0} established some results on existence of best proximity points of cyclic contractions in the context of metric space and utilized the same to discuss the existence, uniqueness and convergence of  best proximity points in the framework of uniformly convex Banach spaces. In the same continuation,  Basha \cite{SSB-1,SSB-2,SSB-3} extended the Banach contraction principle employing the minimum distance between two subsets of the metric space and utilized the same to prove the best proximity point results for non-self proximal contractions.\\

Indeed, the best proximity point theorems offer an approximate solution that is optimal. However, the best approximation theorems yield approximate solutions that are not necessarily optimal. It is interesting to see that best proximity point theorems generalize fixed point theorems in a natural way. In fact, when the mapping under consideration is a self-mapping, a best proximity point boils down to a fixed point. The theory of best proximity point has the great importance in nonlinear analysis, approximation theory, optimization theory, game theory, fixed point theory and variational inequalities.\\

The aim of this paper is to refine and to improve the proximal contraction principle due to Basha \cite{SSB-3}. In process, we introduce some new notions, such as: $T$-proximal sequence, proximally  completeness, proximally  closedness and proximally continuity. As an application of our result, we study the existence of a unique solution of a variational inequality problem.\\

\section{Preliminaries}

Given a pair $(A,B)$ of nonempty subsets of a metric space $(X,d)$, the following notations will be utilized in our subsequential discussion:
\begin{enumerate}
\item[] $d(A,B):=\inf\{d(x,y):x\in A,y\in B\}$,
\item[] $d(x,B):=\inf\{d(x,y):y\in B\}$, for $x\in A$,
\item[] $A_0:=\{x\in A:d(x,y)=d(A,B),~~\mbox{for some}~y\in B\}$,
\item[] $B_0:=\{y\in A:d(x,y)=d(A,B),~~\mbox{for some}~x\in A\}$.
\end{enumerate}

It can be noted that for each $x\in A_0$, $\exists$ $y\in B_0$ such that $d(x,y)=d(A,B)$ and
conversely, for each $y\in B_0$, $\exists$ $x\in A_0$ such that $d(x,y)=d(A,B)$. Consequently, $A_0$ is nonempty if and only if $B_0$ is nonempty. Also, it is evident that both  $A_0$ and  $B_0$ are nonempty, whenever $A$ intersects $B$.

\begin{definition} \cite{BPP-0} Let $(X,d)$ be a metric space and $(A,B)$ a pair of two nonempty subsets of $X$. An element $\overline{x}\in A$ is called a best proximity point of $T:A\rightarrow B$ if
$$d(\overline{x},T\overline{x})=d(A,B):=\inf\{d(x,y):x\in A,y\in B\}.$$
In other words, we say that $\overline{x}\in A$ is a best proximity point of $T$ if at $\overline{x}$ the function $d(x,Tx)$ attains its global minimum with the value $d(A,B)$.
\end{definition}

\begin{definition}\cite{SSB-3} Let $(X,d)$ be a metric space and $(A,B)$ a pair of nonempty subsets of $X$. A mapping $T:A\rightarrow B$  is called proximal contraction if $\exists$ $k\in [0,1)$ such that for all $x,y,u,v\in A$,
$$d(u,Tx)=d(v,Ty)=d(A,B)\Rightarrow d(u,v)\leq k d(x,y).$$
\end{definition}
\begin{remark} Any proximal contraction mapping is not necessarily continuous. Under the restriction $A=B=X$, the notion of proximal contraction coincides with that of usual contraction.
\end{remark}

\begin{definition}\cite{BPP-0}
Let $(X,d)$ be a metric space and $(A,B)$ a pair of nonempty subsets of $X$. The set $B$ is called approximatively compact with respect to $A$ if every sequence $\{y_n\}\subset B$ satisfying the condition that $d(x,y_n)\stackrel{\mathbb{R}}\to d(x,B)$ for some $x\in A$,
has a convergent subsequence.
\end{definition}

\begin{remark}\cite{SSB-3} Every set is approximatively compact with respect to itself. Also, every compact set is approximatively compact.
\end{remark}

\begin{remark}\cite{BPP-0} The sets $A_0$ and $B_0$ are nonempty if $A$ is compact and $B$ is approximatively compact with respect to $A$.
\end{remark}

\begin{theorem}\label{thm-00}\cite{SSB-3}
Let $(X,d)$ be a complete metric space, $(A,B)$ a pair of nonempty subsets of $X$ and $T:A\rightarrow B$ a mapping. Suppose that the following conditions hold:
\begin{enumerate}
\item[{\rm(i)}] $A_0\neq \emptyset$ and $B_0\neq \emptyset$,
\item[{\rm(ii)}] $T(A_0)\subseteq B_0$,
\item[{\rm(iii)}] $T$ is  proximal contraction,
\item[{\rm(iv)}] $A$ and $B$ are closed subspaces of $X$,
\item[{\rm(v)}] $B$ is approximatively compact with respect to $A$.
\end{enumerate}
Then $T$ has a unique best proximity point. Further, for any fixed element $x_0\in A_0$, the sequence $\{x_n\}$, defined by
$$d(x_{n+1},Tx_n)=d(A,B)$$
converges to the unique best proximity point of $T$.
\end{theorem}

\section{Main Results}
In this section, we first introduce several metrical notions in our setting and then establish the result on existence and uniqueness of best proximity points.\\

\begin{definition} Let $(X,d)$ be a metric space, $(A,B)$ a pair of nonempty subsets of $X$ and $T:A\rightarrow B$ a mapping. A sequence $\{x_n\}\subset A$ is called $T$-proximal if
$$d(x_{n+1},Tx_n)=d(A,B).$$
\end{definition}

\begin{remark} In particular for $A=B=X$,  the notion of
 $T$-proximal sequence coincides with that of sequence of Picard iteration.
\end{remark}

\begin{definition} Let $(X,d)$ be a metric space, $(A,B)$ a pair of nonempty subsets of $X$ and $T:A\rightarrow B$ a mapping. We say that the subspace $(A,d)$ is proximally  complete if every $T$-proximal Cauchy sequence in $A$ converges in $A$.
\end{definition}
Clearly, every complete subspace of a metric space is proximally complete.\\

\begin{definition} Let $(X,d)$ be a metric space, $(A,B)$ a pair of nonempty subsets of $X$, $T:A\rightarrow B$ a mapping and $E\subseteq A$. We say that $E$ is proximally closed subspace of $A$ if the limit of each $T$-proximal convergent sequence in $E$ belongs to $E$.
\end{definition}
Clearly, every closed subspace of a metric space is proximally closed.

\begin{proposition} Let $(X,d)$ be a metric space, $(A,B)$ a pair of nonempty subsets of $X$ and $T:A\rightarrow B$ a mapping. If $B$ is approximatively compact with respect to $A$, then $A_0$ is proximally closed.
\end{proposition}
\begin{proof} Let $\{x_n\}\subset A_0$ be a $T$-proximal sequence converging to $x\in A$. We have to show that $x\in A_0$. As $\{x_n\}$ is  $T$-proximal, we have
$$d(x_{n+1},Tx_n)=d(A,B).$$
Further since $T(x_n)\in B$, therefore we have
\begin{eqnarray*}
d(x,B)&\leq& d(x,Tx_n)\\
&\leq&d(x,x_{n+1})+d(x_{n+1},Tx_n)\\
&=&d(x,x_{n+1})+d(A,B)\\
&\leq&d(x,x_{n+1})+d(x,B)
\end{eqnarray*}
so that
\begin{equation}\label{eqn-00}
d(x,B)\leq d(x,Tx_n)\leq d(x,x_{n+1})+d(x,B).
\end{equation}
Letting $n\to\infty$, inequality \eqref{eqn-00} gives rise to
$$d(x,Tx_n)\stackrel{\mathbb{R}}\to d(x,B).$$
By approximatively compactness of $B$, $\{Tx_n\}$ has a subsequence $\{Tx_{n_k}\}$ converging to $y\in B$. Thus, we have
$$d(x,y)=d(A,B)$$
yielding thereby $x\in A_0$. Hence, $A_0$ is proximally closed.
\end{proof}

\begin{definition} Let $(X,d)$ be a metric space and $(A,B)$ a pair of nonempty subsets of $X$. A mapping $T:A\rightarrow B$ is called proximally continuous at a point $x\in A$
if for any $T$-proximal sequence $\{x_n\}\subset A$ such that
$x_n\stackrel{d}{\longrightarrow} x$, we have
$T(x_n)\stackrel{d}{\longrightarrow} T(x)$. $T$ is called proximally
continuous if it is proximally continuous at each
point of $X$.
\end{definition}
Clearly, every continuous function is proximally continuous.\\

Now, we present a sharpened version of  Theorem \ref{thm-00}.  Our result improves Theorem \ref{thm-00} in
the following respects:
\begin{enumerate}
\item[{$\bullet$}] ``Completeness of whole metric space $X$" is replaced by ``proximally completeness of subspace $A$". Moreover, closedness of $A$ and $B$ can be relaxed.
\item[{$\bullet$}] ``Approximatively compactness of $B$ (w.r.t. $A$)" is replaced by relatively weaker
notion, namely, ``proximally closedness of $A_0$". Moreover, this condition is not necessary as it can
alternately be replaced by the ``proximally continuity of $T$".
\end{enumerate}

\begin{theorem}\label{main-thm}
Let $(X,d)$ be a metric space, $(A,B)$ a pair of nonempty subsets of $X$  and $T:A\rightarrow B$ a mapping. Suppose that the following conditions hold:
\begin{enumerate}
\item[{\rm(i)}] $A_0\neq \emptyset$
\item[{\rm(ii)}] $T(A_0)\subseteq B_0$,
\item[{\rm(iii)}] $T$ is  proximal contraction,
\item[{\rm(iv)}] $A$ is proximally complete subspace of $X$,
\item[{\rm(v)}] either $T$ is  proximally continuous or $A_0$ is proximally closed subspace of $A$.
\end{enumerate}
Then $T$ has a unique best proximity point. Further, for any fixed element $x_0\in A_0$, the $T$-proximal sequence $\{x_n\}$ based on initial point $x_0$ converges to the unique best proximity point of $T$.
\end{theorem}
\begin{proof} In view of assumption (i), $\exists$ $x_0\in A_0$. As $T(x_0)\in T(A_0)\subseteq B_0$, $\exists$ $x_1\in A_0$ such that
\begin{equation*}
d(x_1,Tx_0)=d(A,B).
\end{equation*}
As $T(x_1)\in T(A_0)\subseteq B_0$, $\exists$ $x_2\in A_0$ such that
\begin{equation*}
d(x_2,Tx_1)=d(A,B).
\end{equation*}
Continuing this process, by induction, we can construct a sequence $\{x_n\}\subset A_0$ such that
\begin{equation}\label{eqn-1}
d(x_{n+1},Tx_n)=d(A,B).
\end{equation}
Hence, $\{x_n\}$ is a $T$-proximal sequence.
Applying the contractivity condition (iii) to \eqref{eqn-1}, we deduce, for all $n\in \mathbb{N}
 _0$ and for some $k\in [0,1)$, that
$$d(x_{n+1},x_{n+2})\leq k d(x_{n},x_{n+1}),$$
which by induction yields that
\begin{equation}\label{eqn-2}
d(x_{n+1},x_{n+2})\leq k^{n+1} d(x_0,x_1).\end{equation}
For all $m,n\in\mathbb{N}$ with $m<n$, using \eqref{eqn-2} and triangular inequality, we get
 \begin{eqnarray*}
 \nonumber d(x_m,x_n)&\leq& d(x_{m},x_{m+1})+d(x_{m+1},x_{m+2})+\cdots+d(x_{n-1},x_{n})\\
 &\leq&(k^{m}+k^{m+1}+\cdots+k^{n-1})d(x_0,x_1)\\
 &=&k^{m}(1+k+k^{2}+\cdots+k^{n-m-1})d(x_0,x_1)\\
 &\leq& \frac{k^m}{1-k}d(x_0,x_1),~~({\rm where}~ 0\leq k<1)\\
 &\rightarrow& 0\;{\rm as}\;m\;({\rm and~hence}~n)\rightarrow \infty,
\end{eqnarray*}
which implies that the sequence $\{x_n\}$ is Cauchy. Hence, $\{x_n\}$ is a $T$-proximal Cauchy sequence in $A$. By proximally
completeness of $A$, $\exists~\overline{x}\in A$ such that
$x_n\stackrel{d}{\longrightarrow} \overline{x}$.\\

Now, we use assumption (v) to show that $x$ is a best proximity
point of $T$. Suppose that $T$ is proximally  continuous. As
$\{x_n\}$ is a $T$-proximal sequence satisfying
$x_n\stackrel{d}{\longrightarrow} \overline{x}$, proximally continuity of
$T$ implies that $T(x_n)\stackrel{d}{\longrightarrow} T(\overline{x})$.
Using continuity of $d$ and \eqref{eqn-1}, we get
$$d(\overline{x},T\overline{x})=d(\lim\limits_{n\rightarrow\infty} x_{n+1},\lim\limits_{n\rightarrow\infty} Tx_n)=\lim\limits_{n\rightarrow\infty} d(x_{n+1},Tx_n)=d(A,B)$$
so that $\overline{x}$ is a best proximity point of $T$. Alternately, assume that $A_0$ is proximally closed subspace of $A$. As
$\{x_n\}\subset A_0$ is a $T$-proximal sequence satisfying
$x_n\stackrel{d}{\longrightarrow} \overline{x}\in A$, we have $\overline{x}\in A_0$. Also by assumption (ii), we get $T(\overline{x})\in B_0$. Therefore, $\exists~\omega\in A_0$ such that
\begin{equation}\label{eqn-5}
d(\omega,T\overline{x})=d(A,B).
\end{equation}
Combining  \eqref{eqn-1} and \eqref{eqn-5}, we get
\begin{equation}\label{eqn-6}
d(x_{n+1},Tx_{n})=d(\omega,T\overline{x})=d(A,B).
\end{equation}
Using assumption (iii) and
$x_{n}\stackrel{d}{\longrightarrow} x$, we obtain for some $0\leq k<1$ that
\begin{eqnarray*}
 \nonumber d(x_{n+1},\omega)\leq k d(x_{n},\overline{x})\rightarrow 0\;{\rm as}\; n\rightarrow \infty
\end{eqnarray*}
so that $x_{n}\stackrel{d}{\longrightarrow}\omega$. Owing to the uniqueness of limit, we obtain $\omega=\overline{x}$. Hence, (\ref{eqn-5}) reduces to
$$d(\overline{x},T\overline{x})=d(A,B),$$
so that $\overline{x}$  is a best proximity point of $T$.\\

Finally, we prove the uniqueness of  best proximity point. Suppose that  $\overline{x}$ and $\overline{y}$  are two best proximity points of $T$. Then, we have
\begin{equation}\label{eqn-7}
d(\overline{x},T\overline{x})=d(\overline{y},T\overline{y})=d(A,B).
\end{equation}
Applying proximal contractivity condition (iii) to \eqref{eqn-7}, we get
$$d(\overline{x},\overline{y})\leq k d(\overline{x},\overline{y}),$$
for some $0\leq k<1$, yielding thereby $\overline{x}=\overline{y}$. Hence, $T$ has a unique best proximity point.
\end{proof}

\begin{remark} Under the restriction $A=B=X$, Theorem \ref{main-thm} reduces to classical Banach contraction principle.
\end{remark}

\section{An Application to Variational Inequality Problem}
Let $H$ we a real Hilbert space with inner product $\langle\cdot,\cdot\rangle$ and induced norm $\|\cdot\|$. Let $K$ be a nonempty closed and convex subset of $H$ and $S:H\to H$ an operator. Consider the following variational inequality problem:
\begin{equation}\label{vi0}
\mbox{Find~} u\in K \mbox{such~that~} \langle Su,v-u\rangle\geq 0,\quad\forall~v\in K.
\end{equation}
To solve problem \eqref{vi0}, we use the theoretical results concerning the metric projection operator $P_K: H\to K$. Notice that that for each $u\in H$, there exists a unique nearest point $P_K(u)$ such that
$$\|u-P_Ku\|\leq \|u-v\|,\quad\forall~v\in K.$$

The theory of variational inequalities is motivated by equilibrium problems. During last three decades, the theory of  variational inequalities emerged as a rapidly growing area of research due to its appearance in the fields of nonlinear analysis, operations research, economics, game theory, mathematical physics and calculus of variations associated with the minimization of
infinite-dimensional functionals. For details see \cite{VI-1,VI-2,VI-3} and the references therein. In the recent past, many authors solved the variational inequality problems employing fixed point theorems. The following known results correlate the solvability of a variational inequality problem to the solvability of certain fixed point problem.

\begin{proposition}
\cite{VI-3} Let $z\in H$. Then $u\in K$ satisfies the inequality $\langle u-z,y-u \rangle\geq 0$,
for all $y\in K$ if and only if $u=P_K(z)$.
\end{proposition}

\begin{lemma}\cite{VI-3}\label{lem} Let $S:H\to H$ an operator. Then $u\in K$ solves the inequality $\langle Su,v-u \rangle\geq 0$, for all $v\in K$ if and only if $u=P_K(u-\lambda Su)$, where $\lambda>0$.
\end{lemma}

The main result of this section runs as follows:

\begin{theorem} 
Let $K$ be a nonempty closed and convex subset of a real Hilbert space $H$. If $S:H\to H$ is an operator such that $P_K(I-\lambda S):K\to K$, with $\lambda>0$, forms a proximal contraction, then the problem \eqref{vi0} has a unique solution $u^\ast\in K$. Moreover, for each $u_0\in K$, the recursive sequence $\{u_n\}\subset K$ defined by $u_{n+1}:=P_K(u_n-\lambda Su_n)$, for all $n\in \mathbb{N}$, converges to $u^\ast$. 
\end{theorem}
\begin{proof}  Define the operator $T:K\rightarrow K$ by $T(x)=P_K(x-\lambda Sx)$, for all $x\in K$. $T$ satisfies all the hypotheses mentioned in  Theorem \ref{main-thm} for $A=B=K$ in the context of real Hilbert space $H$ equipped with norm metric $d$. Consequently, there exists a unique $u^\ast$ such that $d(u^\ast,Tu^\ast)=d(K,K)=0$ and $u_n\stackrel{d}{\longrightarrow}  u^\ast$. This yields that $u^\ast=T(u^\ast)=P_K(u^\ast-\lambda S u^\ast)$ and hence from Lemma \ref{lem}, it follows that $u^\ast$ is a solution of the problem \eqref{vi0}. This completes the proof.
\end{proof}



\end{document}